\documentclass[12pt]{amsart}

\title{Limit points of uniform arithmetic bass notes}
\author{Will Hide}
\address[Will Hide]{Mathematical Institute,
	University of Oxford,
	Andrew Wiles Building, OX2 6GG Oxford,
	United Kingdom}
\email{william.hide@maths.ox.ac.uk}

\author{Bram Petri}
\address[Bram Petri]{Institut de Math\'ematiques de Jussieu--Paris Rive Gauche and  Institut universitaire de France ; Sorbonne Universit\'e and Universit\'e Paris Cit\'e, CNRS, IMJ-PRG, F-75005 Paris, France}
\email{bram.petri@imj-prg.fr}

\date{\today}

\usepackage{amsmath,amsfonts,amssymb,amsthm,graphicx,overpic}
\usepackage{float,enumitem}
\usepackage{color,xcolor}
\usepackage[colorinlistoftodos]{todonotes}
\usepackage{multirow}
\usepackage{pgfplots}
\usepackage{subcaption}
\usepackage[colorlinks=true,linkcolor=red]{hyperref}

\usepackage{tikz-cd}

\pgfplotsset{compat=1.7}

\usepackage[margin=2.7cm]{geometry}  
\usepackage[indent]{parskip}

\numberwithin{equation}{section}

\makeatletter
\def\namedlabel#1#2{\begingroup
    #2%
    \def\@currentlabel{#2}%
    \phantomsection\label{#1}\endgroup
}
\makeatother

\newtheorem{thm}{Theorem}[section]

\newtheorem{prp}[thm]{Proposition}

\newtheorem{lem}[thm]{Lemma}

\newtheorem{mainthm}{Theorem}

\theoremstyle{definition}

\newtheorem{rem}[thm]{Remark}

\newcommand{\nc}{\newcommand}
\nc{\dmo}{\DeclareMathOperator}
\usepackage{dsfont}

\nc{\abs}[1]{\left| #1 \right|}
\nc{\bigO}[1]{O\left(#1\right)}
\nc{\card}[1]{\left|#1\right|}
\nc{\ceil}[1]{\left\lceil #1 \right\rceil}
\nc{\CC}{\mathbb{C}}
\nc{\dilog}{\mathcal{L}}
\nc{\floor}[1]{\left\lfloor #1 \right\rfloor}
\nc{\ind}{\mathds{1}}
\nc{\ZZ}{\mathbb{Z}}
\nc{\len}[1]{\left| #1 \right|}
\nc{\littleo}[1]{o\left(#1\right)}
\dmo{\Mat}{Mat}
\nc{\NN}{\mathbb{N}}
\nc{\norm}[1]{\left|\left| #1 \right|\right|}
\nc{\QQ}{\mathbb{Q}}
\nc{\RR}{\mathbb{R}}
\nc{\st}[2]{\left\{\, #1 \,:\, #2\,\right\}}
\dmo{\supp}{supp}
\nc{\tr}[1]{\mathrm{tr}\left(#1\right)}
\nc{\what}{\widehat}
\dmo{\im}{Im}
\nc{\eps}{\varepsilon}
\dmo{\li}{li}
\dmo{\arccosh}{arccosh}

\dmo{\area}{area}
\dmo{\conv}{conv}
\dmo{\diam}{diam}
\dmo{\DD}{\mathbb{D}}
\dmo{\dist}{\mathrm{d}}
\nc{\HH}{\mathbb{H}}
\dmo{\Isom}{Isom}
\dmo{\MCG}{MCG}
\dmo{\MPL}{MPL}
\dmo{\Mod}{\mathcal{M}}
\dmo{\PL}{PL}
\nc{\Sphere}{\mathbb{S}}
\dmo{\sys}{sys}
\dmo{\kiss}{Kiss}
\dmo{\Teich}{\mathcal{T}}
\nc{\Torus}{\mathbb{T}}
\dmo{\vol}{vol}
\dmo{\WP}{WP}

\dmo{\convTV}{\;\stackrel{\mathrm{TV}}{\longrightarrow}\;}
\nc{\ExV}[2]{\mathbb{E}_{#1}\left[#2\right]}
\dmo{\EE}{\mathbb{E}}
\nc{\Pro}[2]{\mathbb{P}_{#1}\left[#2\right]}
\dmo{\PP}{\mathbb{P}}
\nc{\distTV}[2]{\mathrm{d}_{\rm TV}\left(#1,#2\right)}
\dmo{\UU}{\mathbb{U}}
\nc{\Var}[2]{\mathbb{V}\mathrm{ar}_{#1}\left[#2\right]}

\dmo{\alt}{\mathfrak{A}}
\dmo{\Aut}{Aut}
\dmo{\Fix}{Fix}
\dmo{\GL}{GL}
\dmo{\Hom}{Hom}
\dmo{\id}{Id}
\dmo{\PGL}{PGL}
\dmo{\PSL}{PSL}
\dmo{\PO}{PO}
\dmo{\Rep}{Rep}
\dmo{\SL}{SL}
\dmo{\SO}{SO}
\dmo{\sym}{\mathfrak{S}}
\dmo{\inv}{\mathcal{I}}
\dmo{\orb}{\mathcal{O}}
\dmo{\stab}{Stab}

\dmo{\calA}{\mathcal{A}}
\dmo{\calB}{\mathcal{B}}
\dmo{\calC}{\mathcal{C}}
\dmo{\calD}{\mathcal{D}}
\dmo{\calE}{\mathcal{E}}
\dmo{\calF}{\mathcal{F}}
\dmo{\calG}{\mathcal{G}}
\dmo{\calH}{\mathcal{H}}
\dmo{\calI}{\mathcal{I}}
\dmo{\calJ}{\mathcal{J}}
\dmo{\calK}{\mathcal{K}}
\dmo{\calL}{\mathcal{L}}
\dmo{\calM}{\mathcal{M}}
\dmo{\calN}{\mathcal{N}}
\dmo{\calO}{\mathcal{O}}
\dmo{\calP}{\mathcal{P}}
\dmo{\calQ}{\mathcal{Q}}
\dmo{\calR}{\mathcal{R}}
\dmo{\calS}{\mathcal{S}}
\dmo{\calT}{\mathcal{T}}
\dmo{\calU}{\mathcal{U}}
\dmo{\calV}{\mathcal{V}}
\dmo{\calW}{\mathcal{W}}
\dmo{\calX}{\mathcal{X}}
\dmo{\calY}{\mathcal{Y}}
\dmo{\calZ}{\mathcal{Z}}

\nc{\klav}{Klav\v{z}ar}

\nc{\bi}{\mathbf{i}}
\nc{\bj}{\mathbf{j}}
\nc{\bk}{\mathbf{k}}

\begin{document}

\begin{abstract} We prove that the set of limit points of the set of all spectral gaps of closed arithmetic hyperbolic surfaces equals $[0,\frac{1}{4}]$.
\end{abstract}

\maketitle

\begin{figure}[h]
\begin{center}
\begin{overpic}{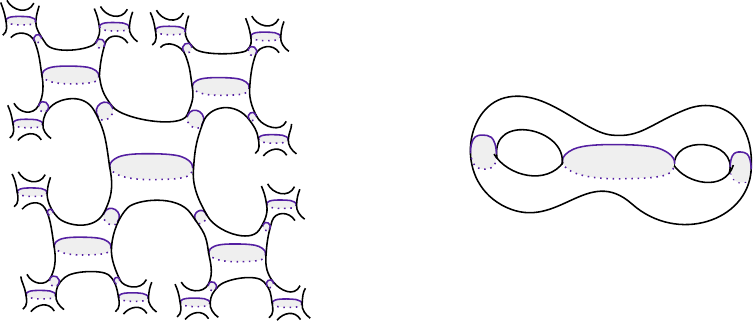}
\put(45,22){{\Huge $\longrightarrow$}}
\end{overpic}
\caption{A tree cover of a surface of genus two.}\label{pic_tree_cover}
\end{center}
\end{figure}

\section{Introduction}

The \textbf{spectral gap} $\lambda_1(X)$ of a closed orientable hyperbolic surface $X$ is the minimal non-zero eigenvalue of its Laplacian. A natural question, that fits into a larger context of the study of the bass note spectra of locally uniform geometries \cite{Sarnak_Chern_lectures}, is which numbers can appear as spectral gaps of hyperbolic surfaces. It can be derived in various ways from various combinations of \cite{Cheng, Gage,Jenni,Bavard,StrohmaierUski, Bonifacio,KravchukMazacPal,FBP_LP} that the \textbf{bass note spectrum} of the set of closed orientable hyperbolic surfaces satisfies:
\[
\mathrm{Bass}(\mathbf{hyp}_2^c) := \st{\lambda_1(X)}{X \text{ a closed orientable hyperbolic surface}} = (0,\Lambda_2],
\]
where $\Lambda_2 = \sup\st{\lambda_1(X)}{X \text{ a closed hyperbolic surface of genus }2}$. The latter conjecturally equals the spectral gap of the Bolza surface (see Section \ref{sec_Bolza} for a description).

If we restrict to arithmetic surfaces, the set of spectral gaps we obtain becomes countable. Recently, Magee \cite{Magee} proved that the spectral gaps -- defined as the infimum of the non-zero spectrum in this setting -- of non-compact arithmetic surfaces form a dense set in $[0,\frac{1}{4}]$. He also asked whether the same is true for closed arithmetic surfaces. The goal of this article is to answer this question in the affirmative:

\begin{mainthm}\label{thm_main}
The set $\Big\{\lambda_1(X):\;X\text{ a closed arithmetic surface}\Big\}\bigcap \left[0,\frac{1}{4}\right]$
is dense in $\left[0,\frac{1}{4}\right]$.
\end{mainthm}

In fact, we show that, like in the non-compact case, it suffices to consider only finite index torsion free subgroups of a single uniform arithmetic group. In our case this is the $(2,3,8)$-triangle group. 
 
Similarly, the same method can be used to prove that if $X=\Gamma\backslash\HH^2$ is a closed orientable surface then
\[
\Big\{\lambda_1(\Gamma'\backslash\HH^2):\;\Gamma'<\Gamma,\; [\Gamma:\Gamma'] < \infty \Big\}\bigcap \left[0,\Lambda\right]
\]
is dense in $[0,\Lambda]$, where
$
\Lambda = \min\st{\lambda_1(\Gamma'\backslash\HH^2)}{\Gamma'<\Gamma,\; [\Gamma:\Gamma']=2}
$.

Finally, we also obtain a different proof of \cite[Theorem 1.8]{HideLouderMagee}, the fact that there exists a sequence of finite sheeted covers $Y_n\to X$ of a given closed orientable arithmetic hyperbolic surface $X$, whose genera tend to infinity and whose spectral gap tends to $\frac{1}{4}$.

\subsection{Strategy of proof}

Like Magee's proof in the non-uniform case, we build our arithmetic surfaces as covers of degree two of a certain type of random covers of a fixed base surface. In our case, this base surface is the Bolza surface $X_B$. As such, the overarching strategy of our proof is similar to that of Magee. However, several of the key ingredients of his proof are not present in our setting, so in what follows we will describe the similarities and the differences.

The first thing that makes our case different, is that the fundamental group $\Gamma_B=\pi_1(X_B)$ of the $X_B$ is not free, which makes random finite degree covers harder to control (see \cite{MageePuder,MageeNaudPuder} for the geometry of uniform random covers of a closed hyperbolic surface). We remedy this by considering random covers $Y^{(n)}\to X_B$ of degree $n$ whose mondromy factors through an epimorphism from our surface group $\Gamma_B\to F_2$ to a free group $F_2$ on two generators. This allows us to use the results of Bordenave--Collins \cite{BordenaveCollins} on strong convergence of permutation representations combined with results by Hide--Magee \cite[Appendix]{HideLouderMagee} on the convergence of the spectral gaps of the associated covers. In \cite{HideLouderMagee}, Louder--Magee also did this. However, we need more control over the geometry of the associated cover, so we use a different argument.

We need to control the following things:
\begin{enumerate}
\item In order to obtain near optimal spectral gaps, we need to understand the spectral gap of the cover of the Bolza surface associated to the map $\Gamma_B\to F_2$ (see Figure \ref{pic_tree_cover}). We tackle this using Cheeger's inequality. Using Mirzakhani's result on random pants decompositions of hyperbolic surfaces \cite{Mirzakhani2} combined with some further geometric arguments, we show that by varying the map to the free group (see Section \ref{sec_bass_notes_trees} and Proposition \ref{prp_good_tree_cover}) we can obtain a spectral gap as close to $\frac{1}{4}$ as we like. 
\item To reach intermediate spectral gaps, our aim is to use the covers of degree two of our random cover $Y^{(n)}\to X_B$. To make this work, we need there to be at least one cover among these degree two covers that has a large spectral gap. We obtain such a cover by restricting a non-constant homomorphism $\Gamma_B\to F_2\to\ZZ/2\ZZ$ to the fundamental group of $Y^{(n)}$, seen as a subgroup of $\Gamma_B$. The fact that this indeed yields two sheeted covers of $Y^{(n)}$  with large spectral gaps relies on a strong convergence argument combined with the fact that all connected two sheeted covers $\what{Y}\to X_B$ satisfy $\lambda_1(\what{Y})>\frac{1}{4}$. The latter fact we prove (in Section \ref{sec_bass_notes_of_Bolza_covers}) using the linear programming methods developed in \cite{FBP_LP,FBGMPP}.
\item Like Magee, we then need to show that, when $n$ is large enough, ``small'' modifications of the degree two covers of $Y^{(n)}$ do not influence the spectral gap too much. These modifications consist of changing the cover along the lifts of a simple closed curve in $Y^{(n)}$ and will be called switches (see Section \ref{sec_switch}). In Magee's paper, logarithmic tangle freeness is used to prove that similar moves don't influence the spectral gap too much. Unfortunately, our covers are, by construction, much less tangle free: the infinite cover associated to the map $\Gamma_B\to F_2$ contains many short curves. As a result, the injectivity radius of $Y^{(n)}$ is uniformly bounded. 

We will however show, again relying on Mirzakhani's results on random pants decompositions and also using Nica's work on cycles of random permutations coming from word maps \cite{Nica}, that we can find maps $\Gamma_B\to F_2$ such that, with positive probability, the curves in $Y^{(n)}$ along which we switch, admit uniformly wide collars (see Section \ref{sec_finishing}). This suffices to make the argument, based on the $L^\infty$-bounds on eigenfunctions by Gilmore--Le Masson--Sahlsten--Thomas \cite{GLMST}, work. This implies the density result because any two covers of degree two of $Y^{(n)}$ can be connected by a finite sequence of switches.
\end{enumerate}

\subsection{Further notes and references} 

The study of spectral gaps of hyperbolic surfaces has a long history, which we will not attempt to summarize here. We refer to \cite{Buser,Bergeron} for introductory texts. The shape of the bass note spectrum has been studied in the case of compact hyperbolic $2$-orbifolds \cite{Bonifacio,KravchukMazacPal}, non-compact arithmetic surfaces \cite{Magee} and regular graphs \cite{AlonWei,DongMcKenzie}. In the case of orbifolds, Kravchuk--Mazac--Pal give a complete description of the set 
\[\st{\lambda_1(O)}{O\text{ a compact orientable hyperbolic }2\text{-orbifold}}.
\]
This description is conjectural, but a large part of it is proven (see \cite[Conjecture 4.2]{KravchukMazacPal}).

The case of arithmetic surfaces is very different. There are only finitely many arithmetic hyperbolic surfaces of any given genus (up to isometry). As such, the theorem above can be stated as saying that the set of limit points of the set of spectral gaps of closed hyperbolic surfaces equals $[0,\frac{1}{4}]$. There are however isolated points outside of this interval, for instance corresponding to the Bolza surface and the Klein quartic. Sarnak \cite{Sarnak_Chern_lectures} conjectured that there should be infinitely many such isolated points.

\subsection*{Acknowledgements}
We are very grateful to Maxime Fortier Bourque for his part in developing the code that we use in Section \ref{sec_bass_notes_of_Bolza_covers}. We also thank the \texttt{SageMath} \cite{sagemath}, \texttt{Arb} \cite{Arb} and \texttt{GAP} \cite{GAP} developers for making the calculations using that same code possible.

We thank the CIRM for hosting the workshop \emph{Random hyperbolic surfaces and random graphs}, part of the Jean-Morlet Chair semester \emph{Random Matrices, Representation Theory and Quantum Information}, during which this project started. 

Finally, we thank Michael Magee and Frédéric Naud for useful conversations surrounding this project.

\subsection*{Funding}
BP was partially supported by the grant ANR-23-CE40-0020-02 ``MOST''.

\section{The model}\label{sec_model}

In this section we formally describe our model of random surfaces, which is a variation of the models appearing in \cite{BCP_diameter,HideLouderMagee,Mathien}.

\subsection{The Bolza surface}\label{sec_Bolza}

The model we will describe samples a certain type of covers of a closed hyperbolic surface. The model makes sense for any closed base surface, but because of the application we have in mind, we will assume the base surface is the Bolza surface. We will use this subsection to describe some of its properties.

The quickest way to define the Bolza surface -- that we shall denote $X_B$ -- is to say that it's a Riemann surface of genus $2$ with $48$ automorphisms. Indeed, it turns out that this is the maximal possible number of automorphisms of a Riemann surface of genus $2$ and there is a unique such surface. More concretely, $X_B$ can be obtained by taking a regular octagon $O$ with interior angles $\pi/4$ in the hyperbolic plane $\HH^2$ and identifying opposite sides. We will also fix a Fuchsian group $\Gamma_B$ such that $X_B = \Gamma_B \backslash \HH^2$ and we will moreover fix an identification of $\Gamma_B$ with $\pi_1(X_B)$.

The surface $X_B$ is also a $(2,3,8)$-triangle surface. That is, its fundamental group is a normal subgroup (of index $48$) in the orientable triangle group 
\[
\Theta(2,3,8) = \langle x,y,z|\; x^2,\; y^3,\; z^8,´\; xyz\rangle
\]
whose generators act on $\HH^2$ by the order $2$, $3$ and $8$ rotations in the vertices of a hyperbolic triangle whose interior angles are $\pi/2$, $\pi/3$ an $\pi/8$ respectively. The quotient group is described explicitly as item \texttt{T2.1} in \cite{Conder}. Since the $(2,3,8)$-triangle group is arithmetic \cite{Takeuchi} (see also \cite[Section 13.3]{MaclachlanReid}), $X_B$ is arithmetic.

Very good numerical estimates of the spectral gap of $X_B$ have been obtained in \cite{StrohmaierUski}. That is,
\[
\lambda_1(X_B) = 3.838887 \pm 10^{-6}.
\]
The best known upper bound on the maximal spectral gap of a hyperbolic surface of genus $2$ equals $3.8388976481$ \cite{Bonifacio,KravchukMazacPal}.

\subsection{Handlebody attachments, maps to free groups and pants decompositions}

A handlebody $H$ is a three-manifold that can be obtained by attaching $1$-handles to a closed $3$-ball. Alternatively it can be seen as the closed regular neighborhood of a finite graph in $\RR^3$. In particular, its fundamental group is a free group and its boundary is a closed orientable surface. The genus of this surface will be called the \textbf{genus} of $H$. This genus coincides with the rank of the fundamental group of $H$. The inclusion map $\iota:\partial H\hookrightarrow H$ induces an epimorphism
\[
\iota_*:\pi_1(\partial H) \longrightarrow F_g,
\]
where $g$ is the genus of $H$. In fact, we obtain a large family of such maps, because we can precompose the inclusion by a homeomorphism $\phi:\partial H \to \partial H$.

Any such homeomorpism yields an $F_g$-cover
\[
T_\phi \to \partial H
\]
corresponding to the subgroup $\ker((\iota\circ\phi)_*)$. The surface $T_\phi$ is what is sometimes called a Cantor tree: a surface homemorphic to $\Sphere^2-C$, where $C$ is a Cantor set. As such, we will call the $F_g$-covers constructed by the procedure above \textbf{tree covers}.

Finally, we observe that if $\calP$ is a pants decomposition of $\partial H$ consisting entirely of \textbf{meridians} -- simple closed curves that bound disks in $H$ -- then $\phi^{-1}(\calP)$ lifts to a pants decomposition of $T_\phi$. For a picture of this set-up, see Figure \ref{pic_tree_cover}.

\subsection{Random covers}\label{sec_rand_cov}

We will now once and for all fix a handlebody attachment $\iota:X_B\hookrightarrow H$ of the Bolza surface $X_B$ and a pants decomposition $\calP_0=(\alpha_1,\alpha_2,\alpha_3)$ of the $X_B$ that consists entirely of meridians.

Given a homeomorphism $\phi:X_B\to X_B$, we obtain, as in the previous section, a homomorphism $(\iota \circ \phi)_*:\Gamma_B \to F_2$ and a tree cover $T_\phi \to X_B$. The pants decomposition of $T_\phi$ that we obtain by lifting $\phi^{-1}(\calP_0)$ will be called the \textbf{canonical pants decomposition} of $T_\phi$ in what follows. 

Now, given a homomorphism $\rho\in\Hom(F_2,\sym_n)$, where $\sym_n$ denotes the symmetric group on $n$ letters, we obtain a subgroup
\[
H^{(n)} = \stab_\rho(\{1\}) \quad < \quad F_2,
\]
but also a homomorphism
\[
\rho\circ (\iota \circ\phi)_*:\Gamma_B \longrightarrow \sym_n
\]
and hence a subgroup
\[
\Lambda_{\phi,\rho} = \stab_{\rho\circ (\iota \circ\phi)_*}(\{1\}) \quad < \quad \Gamma_B
\]
of index at most $n$ and thus a cover 
\[
Y_{\phi,\rho} \longrightarrow X_B
\]
of degree at most $n$. 

Our random surface $Y_\phi^{(n)}$ will be obtained by taking the permutation representation $\rho \in \Hom(F_2,\sym_n)$ uniformly at random. It follows from results by Dixon \cite{Dixon} that, as $n\to\infty$, the probability that the cover $Y_\phi^{(n)}$ is of degree exactly $n$ tends to $1$.

\section{The bass notes of degree two covers of the Bolza surface}\label{sec_bass_notes_of_Bolza_covers}

In this section we will explain how we prove, with the help of a computer, that all the degree two covers $Y\to X_B$ of the Bolza surface satsify $\lambda_1(Y) > \frac{1}{4}$. To do so, we will employ the linear programming methods from \cite{FBP_LP,FBGMPP}.

In fact, just like in \cite{Magee}, this can be proved by verifying the spectral gap of a single surface $Y_{17}$ obtained as the $\mathrm{H}_1(X_B,\ZZ/2\ZZ)$-cover of $X_B$. Indeed, any homomoprhism $\Gamma_B\to\ZZ/2\ZZ$ factors through $\mathrm{H}_1(X_B,\ZZ/2\ZZ)$ and as such any degree two cover of $X_B$ is intermediate to the cover $Y_{17}\to X_B$.

\subsection{Identifying the cover}

First of all, we need a description of $Y_{17}$. Since 
\[
\card{\mathrm{H}_1(X_B,\ZZ/2\ZZ)}=16,
\]
the genus of $Y_{17}$ equals $17$ (whence the notation). Moreover, by construction, the fundamental group of $Y_{17}$ is a characteristic\footnote{Recall that a characteristic subgroup of $\Gamma_B$ is a subgroup that is preserved by all elements of $\Aut(\Gamma_B)$. This is a stronger assumption than being a normal subgroup.} subgroup of $\Gamma_B$. This implies that it's a normal subgroup of the $(2,3,8)$-triangle group $\Theta(2,3,8)$. In other words, $Y_{17}$ is a $(2,3,8)$-triangle surface of genus $17$. It turns out there is only one such surface, which appears as item \texttt{T17.2} in Conder's list \cite{Conder}. In particular, $Y_{17}$ corresponds to the quotient
\[
G= \langle x,y,z|\; x^2,\; y^3,\; z^8,\; xyz,\; z^2 y x z^2 y x z y^{-1}z^{-1} x z y^{-1} z^{-1} x \rangle
 \]
of $\Theta(2,3,8)$.

\subsection{Excluding small eigenvalues}

In order to prove that $Y_{17}$ has no small eigenvalues, we use the method from \cite{FBGMPP} that we will very briefly describe here. For more details, we refer to loc. cit. and \cite{CornelissenPeyerimhoff}.

First of all, the Deck group $G$ of the orbifold cover $Y_{17} \to S_{2,3,8}$, where $S_{2,3,8}=\Theta(2,3,8)\backslash\HH^2$ denotes the orientable $(2,3,8)$-triangle orbifold, acts on functions on $Y_{17}$. This action commutes with the Laplacian, so we can split the eigenspaces of the Laplacian into $G$-representations.

Our goal is now to exclude small eigenvalues from these representations. To do so, we use a criterion based on the twisted Selberg trace formula corresponding to every such representation. 

We start with some set up. We will write $\mathrm{spec}(Y_{17})$ for the spectrum of the Laplacian on $Y_{17}$, as a multi-set. Moreover, $\mathrm{Irr}(G)$ will denote the finite set of real $G$-representations $\rho:G\to\GL(V^\rho)$. The splitting of the eigenspaces into $G$-representations described above yields a splitting of the spectrum
\[
\mathrm{spec}(Y_{17}) = \bigcup_{\rho\in \mathrm{Irr}(G)} \dim(\rho)\cdot \mathrm{spec}(Y_{17},\rho),
\]
where multiplication of a multi-set by a positive integer means multiplying the multiplicities by that integer and
$\mathrm{spec}(Y_{17},\rho)$ is the spectrum of the Laplacian acting on functions $\HH^2\to V^\rho$ that are equivariant with respect to the action of $\Theta(2,3,8)$ on the left and $\rho(G)$ on the right. This spectrum can also be interpreted in terms of the associated bundle.

For $d>0$, we define $f_d:\RR\to\RR$ by
\[
f_d(x) = \left(\frac{1}{2d}\chi_{[-d,d]}\right)^{*4}(x) = \left\{
\begin{array}{ll}
\frac{1}{12d}\left(4-\frac{3}{2d^2}x^2+\frac{3}{8d^3}\abs{x}^3\right) & \text{if } 0\leq \abs{x} \leq 2d, \\[3mm]
\frac{1}{12d}\left(2-\frac{\abs{x}}{2d}\right)^3 &
\text{if }2d \leq \abs{x} \leq 4d \text{ and} \\[3mm]
0 & \text{otherwise}.
\end{array}
\right.
\]
Here $\chi_{[-d,d]}$ denotes the characteristic function of the interval $[-d,d]$ and ``$*$'' denotes the convolution product.

The Fourier transform $\what{f_d}:\CC\to\CC$ of $f$ is given by
\[
\what{f_d}(y) = \int_{-\infty}^\infty f_d(x) e^{-iy\cdot x}dx = \frac{\sin(d\cdot y)^4}{(d\cdot y)^4}, \quad y \in \CC.
\]

Given $\rho\in\mathrm{Irr}(G)$, we define
\begin{align}\label{eq_geomside}
\calG_d(\rho) & = -\frac{\dim(V^\rho)}{96}\int_{-\infty}^\infty \frac{f'_d(x)}{\sinh(x/2)}\, dx \\ \notag
& \quad + \sum_{[\gamma] \in \calE(\Gamma)} \frac{\tr{\widetilde{\rho}(\gamma)}}{m(\gamma)} \int_0^\infty
\frac{\cosh(x/2)}{\cosh(x)-1+2\sin(\theta(\gamma))^2} \, f_d(x)\,dx 
 \\ \notag
& \quad  + \sum_{[\gamma] \in \calP(\Gamma)} \ell(\gamma)\; \sum_{n\geq 1} \frac{\tr{\widetilde{\rho}(\gamma^n)}}{2\sinh(n\ell(\gamma)/2)} f_d(n\ell(\gamma)),
\end{align}
where
\begin{itemize}
\item $\widetilde{\rho}$ is the composition of $\rho:G\to\GL(V^\rho)$ with the quotient map $\Theta(2,3,8) \to G$,
\item $\calE(\Gamma)$ denotes the set of conjugacy classes of elliptic elements in $\Theta(2,3,8)$,
\item for an elliptic element $\gamma\in\Theta(2,3,8)$, $m(\gamma)$ denotes its order and $\theta(\gamma)$ denotes half the angle of rotation, i.e., is such that $\gamma$ is conjugate to
\[
\left[ \begin{array}{cc} \cos(\theta(\gamma)) & \sin(\theta(\gamma)) \\ - \sin(\theta(\gamma)) & \cos(\theta(\gamma)) \end{array}\right] \in \PSL(2,\RR),
\]
\item  $\calP(\Gamma)$ denotes the set of conjugacy classes of primitive hyperbolic elements in $\Theta(2,3,8)$,
\item for a hyperbolic element $\gamma\in\Theta(2,3,8)$, $\ell(\gamma)$ denotes its translation length on $\HH^2$.
\end{itemize}
We observe that, because $f_d$ has finite support, the sum and the integrals appearing on the right hand side of \eqref{eq_geomside} are all finite.

Given these definitions, the criterion we use is the following:

\begin{prp}[{\cite[Proposition 3.1]{FBGMPP}}]\label{prp_LP} Let $\rho\in\mathrm{Irr}(G)$ and let $\lambda>0$. Suppose that
\[
\what{f}_d\left(\sqrt{\lambda-\frac{1}{4}}\right) > \left\{
\begin{array}{ll}\calG_d(\rho) & \text{if }\rho \text{ is non-trivial} \\
\calG_d(\rho)-\what{f}(i/2) & \text{if }\rho\text{ is trivial},
\end{array}
\right.
\]
for some $d>0$. Then $\lambda \notin \mathrm{spec}(Y_{17},\rho)$.
\end{prp}

We prove the following lower bound on the spectral gap of $Y_{17}$:
\begin{prp}
The criterion from Proposition \ref{prp_LP} hold for the parameter $d=\frac{3}{4}$ for all $\rho\in\mathrm{Irr}(G)$ and for all $\lambda \leq 0.2501$. In particular, 
\[
\lambda_1(Y_{17}) \geq  0.2501.
\]
\end{prp}

\begin{proof} The final proof of this proposition is performed using \texttt{SageMath} and its interface to \texttt{GAP}. Here we will discuss the main ingredients.

\begin{enumerate}
\item We compute the character table of $G$, using \texttt{GAP}.

\item Now we observe that when $d=\frac{3}{4}$, the support $\supp(f_d) = [-3,3]$. As such, we need to know all the hyperbolic conjugacy classes of translation length at most $3$ and all the elliptic conjugacy classes of $\Theta(2,3,8)$ in order to compute $\calG_d(\rho)$. The latter are the just the conjugacy classes of the generators $x$, $y$ and $z$ and their powers. The hyperbolic conjugacy classes we need were determined in \cite{FBGMPP}.

\item The integrals appearing in $\calG_d(\rho)$ are treated using interval arithmetic as implemented in the \texttt{Arb} package.
\end{enumerate}

The claim in our proposition is verified in the \texttt{Jupyter} notebook \linebreak \texttt{computation\_spectral\_gap.ipynb} that is attached to the arXiv version of this article as an ancillary file.
\end{proof}

\begin{rem}
Because $0.2501$ suffices for what follows, we have made no effort to optimize the bound in the proposition above. 
\end{rem}

\section{Bass notes of planar surfaces of infinite area}\label{sec_bass_notes_trees}

Given a hyperbolic surface $X$, we define its \textbf{Cheeger constant} by
\[
h(X) = \inf\st{ \frac{\ell(\partial A)}{\area(A)}}{ \begin{array}{c} A\subset X \text{ compact } \\ \partial A \text{ rectifiable} \end{array}}.
\]
The \textbf{systole} -- the length of the shortest closed geodesic on $X$ -- of a hyperbolic surface $X$ will be denoted $\sys(X)$. We call a surface \textbf{planar} if it has genus $0$. Note that such a surface either has cusps or infinite area.

We claim:
\begin{lem}
Let $X$ be a planar orientable hyperbolic surface without cusps. Then
\[
h(X) \geq 1 - \frac{2\pi}{\sys(X) + 2\pi }.
\]
\end{lem}

\begin{proof}
The case in which the fundamental group of $X$ is elementary -- which in our setting implies that $X$ is either the hyperbolic plane or an annulus $\langle g\rangle\backslash\HH^2$ with $g\in\PSL(2,\RR)$ hyperbolic -- is covered by the results of Adams--Morgan \cite[Theorem 2.2, Lemma 2.3]{AdamsMorgan}.

As such, we may assume the fundamental group of $X$ is non-elementary. We define 
\[
H(X) = \inf\st{ \frac{\ell(\partial A)}{\area(A)}}{\begin{array}{c} A\subset X \text{ compact } \\ \partial A \text{ consists of simple closed geodesics} \end{array}}.
\]
Clearly $h(X) \leq H(X)$. It turns out that
\[
h(X) \geq \frac{H(X)}{H(X)+1}
\]
see \cite[Theorem 7]{Matsuzaki}\footnote{Note that Matsuzaki uses the inverse convention for the Cheeger constant.} and \cite[Proposition 4.7]{Mirzakhani1}.

Because $X$ is planar, any candidate subsurface $A$ for $H(X)$ is necessarily planar as well. Suppose $A$ is an $n$-holed sphere. This means that 
\[
\area(A) = 2\pi \cdot (n-2) \quad \text{and} \quad \ell(\partial A) \geq n \cdot \sys(X).
\]
So, $H(X) \geq \frac{n}{n-2}\frac{\sys(X)}{2\pi} \geq \frac{\sys(X)}{2\pi}$ and hence
\[
h(X) \geq 1 - \frac{2\pi}{\sys(X) + 2\pi }.
\]
\end{proof}

Using Cheeger's inequality \cite{Cheeger} (see also \cite[Section IV.3]{Chavel})
%% Cheegers article deals only with the closed case but the proof is the same for non-compact. I can't find a citeable reference for it written explicitly but have seen it cited in many places as e.g. "the proof also holds for  non-compact complete Riemannian manifolds"
%% Found it!
, we obtain:
\begin{prp}\label{prp_gap_of_tree}
Let $X$ be a planar hyperbolic surface without cusps. Then
\[
\lambda_1(X) \geq \frac{1}{4} \cdot \left(1 - \frac{2\pi}{\sys(X) + 2\pi }\right)^2.
\]
\end{prp}

\section{Eigenfunctions near pants curves}

In this section, we describe the effect of flattening an eigenfunction near a pants curve on the Rayleigh quotient of that function.

\subsection{Delocalization}

One of the inputs we will need, is a delocalization result by Gilmore--Le Masson--Sahlsten--Thomas \cite[Theorem 1.5]{GLMST}:

\begin{thm}[Gilmore--Le Masson--Sahlsten--Thomas]\label{thm_GLMST} Suppose that $X=\Gamma\backslash\HH^2$ is a closed hyperbolic surface of genus $g\geq 2$. Then there exists a universal constant $A>0$ such that
\[
\abs{f(z)}^2 \quad \leq \quad A\cdot e^{-t\cdot \sqrt{\frac{1}{4}-\lambda}} \cdot \card{\st{\gamma\in \Gamma}{\dist_{\HH^2}(z,\gamma \cdot z)\leq t}}
\]
for all $t>0$, $\lambda\in (0,\frac{1}{4})$ and all Laplacian eigenfunctions $f$ of eigenvalue $\lambda$ with $\norm{f}_{L^2}=1$.
\end{thm}

\subsection{Flattening an eigenfunction near a pants curve}

Given the result above, we need a bound on the Rayleigh quotient of a flattened eigenfunction in terms of the $L^\infty$-norm of the eigenfunction. Our main technical input is the Ismagilov--Morgan--Simo--Sigal localization formula \cite[Theorem 3.2]{FKS}:

\begin{thm}[Ismagilov--Morgan--Simo--Sigal]\label{thm_IMSS} Let $M$ be a Riemannian manifold and suppose $\{J_i\}_{i\in\calI}$ is a family of smooth functions $J_i:M\to [0,1]$ such that
\begin{enumerate}
\item $\sum_{i\in\calI} J_i^2 \equiv 1$,
\item on any compact subset $K\subset M$, only finitely many $J_i$ are non-zero, and
\item $\sup_{x\in M} \sum_{i\in\calI} \abs{\nabla J_i(x)}^2 < \infty$.
\end{enumerate}
Then 
\[
\Delta = \sum_{i\in\calI} J_i \Delta J_i - \sum_{i\in\calI} \abs{\nabla J_i}^2.
\]
\end{thm}

Given a simple closed geodesic $\gamma$ in a hyperbolic surface $X$ and $w>0$, we will write
\[
C_w(\gamma) = \st{x\in X}{\dist(x,\gamma)\leq w}
\]
for the \textbf{collar} of width $w$ around $\gamma$.  If $C_w(\gamma)$ is isometric to a cylinder $[-w,w]\times \Sphere^1$ with the Riemannian metric 
\[
ds^2 = d\rho^2 + \ell^2 \cosh^2(\rho) dt^2,
\]
where $\ell$ is the length of $\gamma$, we will call the this collar \textbf{standard}. The collar lemma (see \cite{Keen} and \cite[Chaper 4]{Buser}) states that the collar of width $w(\ell) = \mathrm{arcsinh}\left( \frac{1}{\sinh(\ell/2)}\right)$ is always standard. Note that $w(\ell) \sim e^{-\ell/2}$ as $\ell\to \infty$. For our arguments, we will need to work with (long) geodesics that admit standard collars of uniform width.

We will use the following bound, which is the analogue of \cite[Proposition 8.1.]{Magee} in our setting:

\begin{prp}\label{prp_costofswitching}
For all $w>0$ there exists a constant $\eps>0$ such that the following statement holds. Let $X$ be a closed hyperbolic surface and let $\gamma_1,\ldots,\gamma_k \subset X$ be closed geodesics of length $\ell$ such that the collars of width $w>0$ around $\gamma_1,\ldots,\gamma_k$ are standard and disjoint from each other. Moreover, let $f:X\to \RR$ be an $L^2$-normalized Laplacian eigenfunction of eigenvalue $\lambda>0$. Suppose furthermore that\footnote{This is a technical condition we set to make the bound come out nicer. The regime we will be interested in is when this quantity tends to $0$ in any event.} 
\[
\norm{f}_\infty^2 \cdot \ell < \eps.
\]
Then there exists a smooth function $f':X\to \RR$ such that 
\[
f'\vert_{C_{w/2}(\gamma)} \equiv 0, \quad f'\vert_{X-C_w(\gamma)} = f\vert_{X-C_w(\gamma)} \quad \text{ and } \quad \frac{\langle \Delta f',f'\rangle}{\langle f',f'\rangle} - \lambda \leq  B\cdot \norm{f}_\infty^2 \cdot \ell\cdot k,
\]
where $B=B(w)>0$ is a constant depending only on $w$ and $\eps$.
\end{prp}

\begin{proof} Once and for all, fix a symmetric smooth function $J_0:\RR \to [0,1]$ such that 
\[
J_0(x) = \left\{ \begin{array}{ll}
1 & \text{if } \abs{x} \geq 2, \\[2mm]
0 & \text{if } \abs{x} \leq 1.
\end{array}\right.
\]
Using this function, we define $J:C_w(\gamma) \to \RR$ by
\[
J(x) = J_0\left(\frac{2\cdot \dist(x,\gamma)}{w}\right) 
\]
and define $f':X\to \RR$ by
\[
f'(x) = \left\{ 
\begin{array}{ll}
J(x) \cdot f(x) & \text{if } x\in C_w(\gamma) \\[2mm]
f(x) & \text{otherwise}.
\end{array}
\right.
\]
$f'$ satsifies the first two thirds of the claim, so we need to verify the last third.

We will work with one curve (i.e. $k=1$), because we assume disjointness, the argument in the general case is the same, except that the final errors multiply by $k$.

We now first estimate the numerator of the Rayleigh quotient. Using Theorem \ref{thm_IMSS}, we obtain:
\begin{align*}
\langle\Delta f',f'\rangle - \langle \Delta f,f\rangle & = \int_{C_w(\gamma)} (\Delta J f) \cdot J f - (\Delta  f) \cdot  f \\[2mm]
& =  \int_{C_w(\gamma)} \Big( \abs{\nabla J}^2 + \abs{\nabla (\sqrt{1-J^2})}^2\Big)\cdot  f^2 \\
& \quad -\int_{C_w(\gamma)}  \Delta (\sqrt{1-J^2} f) \cdot \sqrt{1-J^2}f \\[2mm]
& \leq \int_{C_w(\gamma)} \Big( \abs{\nabla J}^2 + \abs{\nabla (\sqrt{1-J^2})}^2\Big)\cdot  f^2 \\[2mm]
& \leq \norm{f}^2_\infty \cdot \int_{C_w(\gamma)} \Big( \abs{\nabla J}^2 + \abs{\nabla (\sqrt{1-J^2})}^2\Big)
\end{align*}
Here we have used the fact that $\Delta$ is a positive operator in order to get from the second to the third line.

Now, using the fact $C_w(\gamma)$ is standard and that the function $J$ depends on the distance to the core curve only, we get that
\begin{multline*}
\int_{C_w(\gamma)} \Big( \abs{\nabla J}^2 + \abs{\nabla (\sqrt{1-J^2})}^2\Big) \\ = \ell\cdot\int_{-w}^{w} \left[\left(\frac{d}{d\rho} J_0\left(\frac{\rho}{w}\right)\right)^2 + \left(\frac{d}{d\rho} \sqrt{1-J_0\left(\frac{\rho}{w}\right)^2}\right)^2\right]\cdot \cosh(\rho) d\rho \\
= \frac{\ell}{w} \int_{-1}^1 \left[ \left(\frac{dJ_0(y)}{dy}\right)^2 + \left(\frac{d}{dy}\sqrt{1-J_0(y)^2}\right)^2\right] \cdot \cosh(w\cdot y) dy  \\
\leq A_1\cdot \frac{\ell \cdot \cosh(w)}{w},
\end{multline*}
where $A_1>0$ is a constant depending on our choice of $J_0$ alone. The division by $w$ in the beforelast line comes from the combination of the chain rule and the substitution rule. All in all, we obtain
\[
\langle\Delta f',f'\rangle \leq \langle \Delta f,f\rangle + A_1 \cdot \norm{f}_\infty^2 \cdot  \ell.
\]
Likewise, for the denominator, we have:
\[
 \langle f',f'\rangle = \langle f,f\rangle - \int_{C_\gamma} (1-J^2) f^2 \geq 1- \norm{f}_\infty^2 \cdot \int_{C_\gamma} (1-J^2) \geq 1- A_2\cdot \norm{f}_\infty^2 \cdot \ell \cdot w \cdot \cosh(w)
\]
for some constant $A_2>0$ depending on $J_0$ only.

This means that, setting $E=A_3\cdot \norm{f}_\infty^2 \cdot \ell$ (where $A_3>0$ is large enough to absorb the previous constants and the factors depending on $w$),
\begin{align*}
\frac{\langle \Delta f',f'\rangle}{\langle f',f'\rangle} \leq \frac{\langle \Delta f,f\rangle + E}{\langle f,f\rangle -E} = \frac{\lambda+E}{1-E} \leq \lambda + A_4\cdot E,
\end{align*}
for some constant $A_4>0$, where we've used the geometric series around $E=0$ and our assumption that  $\norm{f}_\infty^2 \cdot \ell<\eps$. Here we can choose any $\eps \in (0,1/A_3)$.
\end{proof}

\section{Building good surfaces}

The goal of this section is to prove the main result of this article: we can produce sequences of covers of the Bolza surface whose spectral gaps are dense in $[0,\frac{1}{4}]$. We will use a variant of the switching technique from \cite{Magee}, based on the model of random surfaces described in Section \ref{sec_model}.

\subsection{The order of constants}\label{sec_constants}

There are many constants involved in our construction and even though they're uniform, the order in which they're fixed matters. So we will record this order here:

\begin{itemize}
\item[\namedlabel{cst_eta}{\textbf{C1}}] We fix a constant $\eta>0$. Our goal will be to build a set of covers whose spectral gaps are $\eta$-dense in $[0,\frac{1}{4}]$.
\item[\namedlabel{cst_collarwidth}{\textbf{C2}}] We also once and for all fix a parameter $w>0$ (independent of $\eta$).  This number appears as a collar width in Proposition \ref{prp_costofswitching} and determines another parameter $\eps>0$ and a third number $B=B(w)$ that both appear in that same proposition.
\item[\namedlabel{cst_length}{\textbf{C3}}] Given all the parameters above, we fix any $\ell>4w$ such that
\[
\frac{1}{4}\cdot \left(1-\frac{2\pi}{\ell+2\pi}\right)^2 > \frac{1}{4}-\frac{1}{2}\cdot \eta \quad \text{and} \quad \ell\cdot e^{-\ell \cdot \sqrt{\eta}} < \min\{\eps/A,\eta/B\},
\]
where $B=B(w)$ is the constant that appears in Proposition \ref{prp_costofswitching} and $A$ the constant that appears in Theorem \ref{thm_GLMST}.
\end{itemize}

\subsection{Building a good tree cover}

From hereon out the parameters chosen in the previous subsection will be considered fixed. 

We will now first prove that we can find good tree covers:

\begin{prp}\label{prp_good_tree_cover}
The Bolza surface admits a tree cover $T_\phi\to X_B$, coming from a homeomorphism $\phi:X_B\to X_B$ such that the following conditions hold:
\begin{enumerate}
\item\label{item_sys} $\sys(T_\phi) > \ell$,
\item\label{item_collar} the collars of width $w$ around all geodesics in the canonical pants decomposition of $T_\phi$ are all standard, and
\item\label{item_gap} $\lambda_1(T_\phi) > \frac{1}{4}-\frac{1}{2}\cdot \eta$.
\end{enumerate}
\end{prp}

\begin{proof}
First of all, Mirakhani \cite[Theorem 1.2]{Mirzakhani2} determined the asymptotic distribution of the normalized length vector
of a uniform random pants decomposition 
\[
\calP_L \in \st{\calP \in \MCG(X_B)\cdot \calP_0}{\sum_{\alpha\in\calP} \ell_\alpha(X_B) \leq L},
\]
where $\MCG(X_B)$ denotes the mapping class group of $X_B$. The normalized length vector of a pants decomposition is the vector
\[
\overline{\ell}(\calP_L) =   \left( \ell_\alpha(X_B) \Big/\; \sum_{\beta\in\calP_L} \ell_\beta(X_B) \right)_{\alpha \in \calP_L} \in \sigma_3=\st{x\in\RR^3}{\begin{array}{c}x_i\geq 0,\; i= 1,2,3 \\ \text{and}\; x_1+x_2+x_3=1 \end{array}}.
\]
Mirzakhani proved that, as $L\to\infty$, this random vector converges in distribution to a random vector in $\sigma_3$ distributed according to a measure that can be explicitly computed. This was generalized to other multi-curves by Arana-Herrera \cite{AranaHerrera} and Liu \cite{Liu}.

For us, the only thing that matters is that the limit measure is absolutely continuous with respect to the Lebesgue measure with a Radon--Nikodym derivative that is strictly positive in the interior of $\sigma_3$. This implies that for all $L>0$ and all $\delta>0$, we can find a homeomorphism $\phi:X_B\to X_B$ such that the pants decomposition $\phi^{-1}(\calP_0) = (\beta_1,\beta_2,\beta_3)$ satifies
\[
\ell(\beta_i) \geq L \quad \text{and} \quad 1-\delta < \frac{\ell(\beta_i)}{\ell(\beta_j)} < 1+\delta \quad \text{for all }i,j \in\{1,2,3\} 
\]
We claim that the tree cover $T_\phi \longrightarrow X_B$ has the properties we require, as soon as $L > \ell/(1-\delta)$.

We start with item \ref{item_sys}: the systole.  First observe that, if the systole corresponds to a pants curve in the canonical pants decomposition of $T_\phi$, then it automatically satisfies the inequality above. If it does not, the geodesic realizing it contains at least two arcs that each run between two points on one of the boundary components of one of the pairs of pants in the pants decomposition. Indeed, the systole is necessarily simple, so it crosses at least two pairs of pants. As such the path it traces in the tree dual to the canonical pants decomposition has at least two leaves. So the length of the systole is at least twice the length of such an arc. All that remains is to find a lower bound on this length.

To do this, consider the arc $\alpha$ in Figure \ref{pic_pants}.
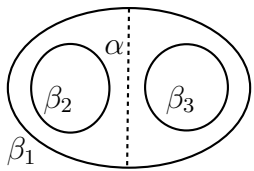
\begin{figure}[!h]
\begin{overpic}{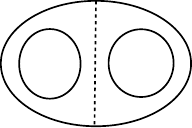}
\put(40,47){$\alpha$}
\put(0,4){$\beta_1$}
\put(15,25){$\beta_2$}
\put(65,25){$\beta_3$}
\end{overpic}
\caption{A pair of pants}\label{pic_pants}
\end{figure}
The arc $\alpha$ cuts the boundary component $\beta_1$ into two arcs, the shortest one of which has length at most $\frac{1+\delta}{2}L$. If we suppose this shortest arc lies on the side of $\beta_2$, we obtain
\[
\ell(\alpha) + \frac{1+\delta}{2}L \geq \ell(\beta_2) \geq L, 
\]
which gives us $\ell(\alpha) \geq \frac{1-\eps}{2}L$ and hence 
\[
\sys(T_\phi) \geq (1-\delta)\cdot L > \ell.
\]

We prove item \ref{item_collar} using essentially the same argument. Indeed, if a collar of width $u$ around one of the pants curves is not standard, then this yields a simple arc of length at most $2u$ between two points on the corresponding pants curve. This arc must be essential relative to its endpoints and by the same argument as above, such an arc has length at least $\ell/2$. So, the collars of width $\ell/4>w$ (by assumption \ref{cst_length}) are standard.

Finally, item \ref{item_gap} follows from the second part of the same assumption \ref{cst_length}, combined with Proposition \ref{prp_gap_of_tree}.
\end{proof}

\subsection{Strong convergence}

To obtain covers of $X_B$ whose spectral gap is nearly optimal we will rely on the techniques developed by Hide--Magee \cite{HideMagee,HideLouderMagee} based on strong convergence of permutation representations.

Let $\Gamma$ be a finitely generated group and suppose $(\calH_n)_n$ is a sequence of Hilbert spaces. Given a sequence of unitary representations $\rho_n:\Gamma \to U(\calH_n)$, where $U(V\calH_n)$ denotes the unitary group of $\calH_n$, we say that this sequence \textbf{strongly converges} to $\rho_{\infty}:\Gamma \to U(\calH_\infty)$ if
\[
\lim_{n\to\infty} \norm{\rho_n(z)} = \norm{\rho_\infty(z)} 
\]
for all $z\in\CC[\Gamma]$.

We will need the following result which is a version of the results in \cite{Magee_survey}, specialized to our set-up:

\begin{thm}\label{thm_HideMagee} Let $\Gamma<\PSL(2,\RR)$ be a uniform lattice and let $\Lambda$ be a group. Moreover suppose that
\begin{itemize}
\item $p:\Gamma \to \Lambda$ is an epimorphism and write $K=\ker(p) \triangleleft \Gamma$, 
\item $\Big(\rho_n\in \Hom(\Lambda,\sym_n)\Big)_{n\geq 1}$ is a sequence of permutation representations such that
\[
\mathrm{std}_n\circ \rho_n \stackrel{\text{strong}}{\longrightarrow} \rho_{\mathrm{reg}}, \quad \text{as }n\to\infty
\]
where $\mathrm{std}_n:\sym_n \to U(\CC^{n-1})$ denotes the $(n-1)$-dimensional irreducible representation and $\rho_{\mathrm{reg}}:\Lambda\to \calB(\ell^2(\Lambda))$ denotes the left regular representation, and
\item write $\Gamma_n = \stab_{\rho_n\circ p}(\{1\}) < \Gamma$
\end{itemize}
Then
\[
\lim_{n\to\infty}\lambda_1(\Gamma_n\backslash\HH^2) = \min\Big\{\lambda_1(K\backslash\HH^2),\lambda_1(\Gamma\backslash \HH^2)\Big\}.
\]
\end{thm}

\begin{proof}[Proof sketch]
The precise statement is not yet available in the literature although it can be derived using small modifications of the arguments in \cite{HideMagee,HideLouderMagee,MageeThomas}. A representation theoretic argument due to Magee \cite{Magee_survey} is perhaps the fastest way to obtain it. We will very briefly describe how this works here.

The fact that $\mathrm{std}\circ \rho_n \stackrel{\text{strong}}{\longrightarrow} \rho_{\mathrm{reg}}$, implies that
\[
\mathrm{std}_n\circ \rho_n \circ p \quad\stackrel{\text{strong}}{\longrightarrow}\quad \rho_{\mathrm{reg}}\circ p : \Gamma \to \calB(\ell^2(\Lambda)).
\]
Magee \cite[Theorem 3.1]{Magee_survey} proves that this implies that:
\[
\mathrm{Ind}_\Gamma^{\PSL(2,\RR)}\Big(\mathrm{std}_n\circ \rho_n \circ p\Big) \stackrel{\text{strong}}{\longrightarrow} \mathrm{Ind}_\Gamma^{\PSL(2,\RR)}\Big(\rho_{\mathrm{reg}}\circ p\Big) .
\]
By \cite[Proposition 2.9]{Magee_survey}, this implies our statement.
\end{proof}

Strongly convergent sequences of representations are notoriously hard to find. However, the following theorem by Bordenave--Collins \cite{BordenaveCollins} gives us what we need:
\begin{thm}[Bordenave--Collins]\label{thm_BordenaveCollins}
Let $F$ be a finitely generated non-abelian free group and let $\rho_n \in \Hom(F,\sym_n)$ be chosen uniformly at random. Then
\[
\mathrm{std}_n\circ \rho_n \stackrel{\text{strong}}{\longrightarrow} \rho_{\mathrm{reg}}, \quad \text{as }n\to\infty
\]
in probability.
\end{thm}

\subsection{Covers of degree two and Hamming distance}

Using good tree covers, combined with the existence of strongly convergent sequences of permutation representations of $F_2$ and Theorem \ref{thm_HideMagee} above, we can already use our construction to build sequences of finite covers of $X_B$ whose spectral gaps are arbitrarily close to $\frac{1}{4}$. As such, we need to explain how to get intermediate spectral gaps, which is what the remainder of the text is devoted to.

From hereon out, we will suppose that $\phi:X_B\to X_B$ is a homeomorphism such that $T_\phi$ satisfies the result from Proposition \ref{prp_good_tree_cover}. Recall that the sequence $\Big(Y_\phi^{(n)}\to X_B\Big)_{n\geq 1}$ denotes the associated sequence of random covers (see Section \ref{sec_model}).

In order to obtain finite covers of $X_B$ whose spectral gaps are $\eta$-dense, we will use a variation of the switching argument of Magee. Magee phrases the argument in terms of bundles over $Y_\phi^{(n)}$ coming from homomoprhisms $\pi_1(Y_\phi^{(n)})\to\ZZ/2\ZZ$. We will phrase it entirely in terms of covers of degree two of $Y_\phi^{(n)}$, which is an equivalent point of view. 

First of all, we will restrict to covers that factor through the handle body attachment. That is, the groups defined in Section \ref{sec_rand_cov} fit in the following commutative diagram:
\begin{center}
\begin{tikzcd}
 \Lambda_\phi^{(n)} \arrow[hookrightarrow]{r}\arrow[d,"(\iota\circ\phi)_*"] & \Gamma_B \arrow[d,"(\iota\circ\phi)_*"] \\
 H^{(n)} \arrow[hookrightarrow]{r} & F_2
\end{tikzcd}
\end{center}
which allows us to define the set
\[
\calH_\phi^{(n)} = \st{\rho\circ(\iota\circ\phi)_*}{\rho\in\Hom(H^{(n)},\ZZ/2\ZZ)}.
\]
Given a free generating set $S=(s_1,\ldots,s_k)$ of $H^{(n)}$, we can identify the set of maps $\calH_\phi^{(n)}$ with $(\ZZ/2\ZZ)^k$ through the map that associates $(\rho(s_1),\ldots,\rho(s_k))$ to $\rho\circ (i\circ\phi)_* \in \calH^{(n)}_\phi$. This in turn allows us to define a \textbf{Hamming distance} $\dist_S:\calH_\phi^{(n)} \times \calH_\phi^{(n)}  \to \NN$ by
\[
\dist_S\Big(\rho_1\circ(\iota\circ\phi)_*,\;\rho_2\circ(\iota\circ\phi)_*\Big) = \card{\st{j}{\rho_1(s_j)\neq \rho_2(s_j)}}.
\]

We note that every element $\rho\in \calH_\phi^{(n)}$ corresponds to a cover of $Y_\phi^{(n)}$. If the image of $\rho$ is trivial, we will let this cover be the cover of $Y_\phi^{(n)}$ by two disjoint copies of itself. This way, all the corresponding covers are of degree two.

\subsection{Switching}\label{sec_switch}

In this section we will describe how two covers of degree two of $Y_\phi^{(n)}$ whose Hamming distance is $1$ differ from each other. 

\begin{figure}[h]
\begin{center}
\begin{overpic}{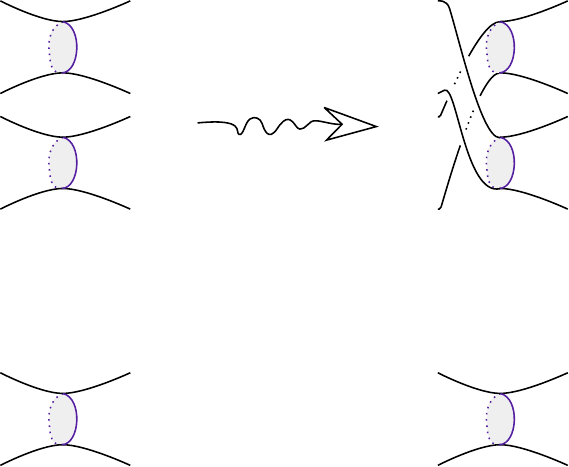}
\put(10.5,30){$\downarrow$}
\put(87,30){$\downarrow$}
\end{overpic}
\caption{A local picture of a switch. The surface below represents a neighborhood of one of the prefered pants curves $\alpha_i^{(n)}$ of $Y^{(n)}_\phi$. The surfaces up top (one on the left and one on the right) are the neighborhoods of the lifts of $\alpha_i^{(n)}$ in two distinct covers of degree two. The shading is there to indicate that the curves that are involved are meridians.}\label{pic_switch}
\end{center}
\end{figure}

To make this precise, we need to first decribe a free generating set of $H^{(n)}$. One way to do this is to identify $F_2$ with the fundamental group of a graph $W^{(1)}$. We will let $W^{(1)}$ be a wedge of two circles. Because the pants curves in $\calP_0$ are meridians, the two edges of $W^{(1)}$ are dual to two of the pants curves in the pants decomposition $\calP_0$ we've fixed. If we change the handlebody attachment by a homeomoprhism $\phi:X_B\to X_B$, this changes the identification. That is, we can think of the edges of $W^{(1)}$ as corresponding to two of the three pants curves in $\phi^{-1}(\calP_0)$. $H^{(n)}$ is the fundamental group of a graph $W^{(n)}$ that covers $W^{(1)}$. The edges in this graph are dual to the canonical pants decomposition of $Y_\phi^{(n)}$. 

In order to obtain a prefered free generating set of $H^{(n)}$, we pick an arbitrary spanning tree in $W^{(n)}$ and contract it. The edges that do not appear in the tree yield a generating set $S^{(n)}=(s_1,\ldots,s_k)$ of $H^{(n)}$. Because of the correspondance between edges and pants curves described above, each of these generators is dual to a pants curve in the canonical pants decomposition of $Y_\phi^{(n)}$. We will denote these prefered pants curves by $(\alpha_1^{(n)},\ldots,\alpha_k^{(n)})$.

Given a map $\rho:H^{(n)}\to \ZZ/2\ZZ$, the corresponding cover of degree two of $Y_\phi^{(n)}$ can be built as follows:
\begin{itemize}
\item Write $Z_\phi^{(n)} = \overline{Y_\phi^{(n)} - \cup_{i=1}^k \alpha_i}$, where $\overline{X}$ denotes the completion of a hyperbolic surface $X$, so $Z_\phi^{(n)}$ is a surface with boundary,
\item take two disjoint copies $Z_1$ and $Z_2$ of $Z_\phi^{(n)}$, and
\item for each pants curve $\alpha_i$: 
\begin{itemize}
\item if $\rho(s_i) = [0]$, glue the two boundary components of $Z_1$ corresponding to $\alpha_i$ back together and do the same for those in $Z_2$,
\item if $\rho(s_i) = [1]$, glue the ``left'' boundary component of $Z_1$ corresponding to $\alpha_i$ to the ``right'' boundary component of $Z_2$ and, likewise, glue the corresponding right boundary component of $Z_1$ to the left boundary component of $Z_2$.
\end{itemize}
All these gluings above should be performed with the correct twist, so as to guarantee that the resulting surface covers $Y_\phi^{(n)}$.
\end{itemize}

It follows from this desciption that, if 
\[
\dist_{S^{(n)}}\Big( \rho_1 \circ (i\circ \phi)_*,\; \rho_2 \circ (i\circ \phi)_*\Big) =1,
\]
then the two corresponding covers can be obtained from each other by opening up the lifts corresponding to one of the prefered pants curves $\alpha_i^{(n)}\subset Y_\phi^{(n)}$ and gluing them back together with the opposite pattern. We will call this operation a \textbf{switch} -- this same terminology is used in graph theory for what happens at the level of the $2$-covers of the graph $W^{(n)}$. Figure \ref{pic_switch} shows a local picture of this move.

\subsection{Finishing the proof}\label{sec_finishing}

We now have everything we need in place to finish the proof of our main theorem. Indeed, Theorem \ref{thm_main} is implied directly by the following result:

\begin{thm}
For all $n\in\NN$ large enough, the probability that the set
\[
\st{\lambda_1(\what{Y})}{\what{Y} \text{ is a connected cover of }Y_\phi^{(n)} \text{ of degree two}}
\]
is $\eta$-dense in $[0,\frac{1}{4}]$, is strictly positive.
\end{thm}

\begin{proof}
First of all, we use Proposition \ref{prp_good_tree_cover} to choose a homeomorphism $\phi:X_B\to X_B$ that gives us a tree cover with the properties in that same proposition.

It now follows from Theorems \ref{thm_HideMagee} and \ref{thm_BordenaveCollins} that 
$\lambda_1(Y_\phi^{(n)}) > \frac{1}{4}-\eta$ with high probability. In fact, using the same argument as in \cite[Section 4]{Magee}, \[
\theta\otimes\rho_n \stackrel{\text{strong}}{\longrightarrow} \rho_{\mathrm{reg}}
\]
in probability as $n\to\infty$, for any $\theta\in\Hom(F_2,\ZZ/2\ZZ)$. For us this means that, again using Theorem \ref{thm_HideMagee}, we obtain covers of degree two, corresponding to some element in $\calH^{(n)}_\phi$, of $Y_\phi^{(n)}$ whose spectral gap is at least 
\[
\lambda_1(T_\phi) - \eta/2 > \frac{1}{4}-\eta
\]
for $n$ large enough. 

The set $\calH^{(n)}_\phi$ also contains the trivial homorphism, the corresponding degree two cover of which has $\lambda_1(\what{Y})=0$. 

We can go between these covers using a finite number of switches, so as soon as we prove that a single switch doesn't change the spectral gap by more than $\eta$, we are done.

In order to do this, we need to make one further assumption on the surfaces $Y^{(n)}$: the preferred pants curves in $Y^{(n)}$ need to have collar width at least $w$ and the systole of $Y^{(n)}$ needs to be at least $\ell$. 

Both of these statements are true in the tree cover $T_\phi$. This means, that if they're not true in $Y^{(n)}$, there needs to be a closed geodesic $\gamma\subset X_B$ of uniformly bounded length $L=L(w,\ell)$ such that
\begin{equation}\label{eq_fixed_point}
\iota_*(\gamma) \neq e \quad \text{but} \quad \rho_n\circ (\iota\circ\phi)_*(\gamma) \cdot 1 = 1.
\end{equation}
It follows from Nica's work \cite{Nica} that, with asymptotically strictly positive probability, $\rho_n(g)$ is fixed point free for all $g$ in any fixed finite set of non-trivial conjugacy classes in $F_2$. The statement we use here is the following. First of all, let $K_1,\ldots K_m\subset F_2$ be a finite collection of distinct non-trivial conjugacy classes, none of which contain a power of a non-trivial element of $F_2$. Moreover, let $(c_1,\ldots,c_k)$ be a finite set of cycle lengths. Then the vector of random variables $\Big(X_{c_i}(K_j)\Big)_{1\leq i \leq m,\; 1\leq j\leq k}$, where $X_{c_i}(K_j)$ denotes the number of $c_i$-cycles in the image of an element of $K_j$ under a uniformly random homomorphism $\rho \in \Hom(F_2,\sym_n)$ converges in distribution to a vector of independent random variables. Furthermore, the limit of $X_{c_i}(K_j)$ is $\mathrm{Poisson}(1/c_i)$-distributed. This statement is not literally contained in Nica's paper, but it can be proven using similar methods. Versions of this statement are also available in \cite[Section 4]{LinialPuder}, \cite[Theorem 1.4]{BakerPetri}, \cite[Corollary 1.7 and Theorem 1.14]{PuderZimhoni}. The only consequence we need is that, there exists a uniform $a>0$ such that for all $n$ large enough:
\[
\PP\Big(X_{c_i}(K_j) = 0\text{ for }i=1,\ldots,k \text{ and } j=1,\ldots,m\Big) > a.
\]
Also observe that if $g = u^k\in F_2$ and if, under a permutation representation of $F_2$, $u$ has no $d$-cycles for any divisor $d$ of $k$, then $g$ has no fixed points.

Since $\Gamma_B$ contains a finite number of conjugacy classes of elements of bounded translation length, we can avoid the existence of a closed geodesic $\gamma\subset X_B$ as in \eqref{eq_fixed_point} with strictly positive probability. Moreover, the existence of a $2$-cover with spectral gap above $\frac{1}{4}-\eta$ has probability tending to $1$ as $n\to \infty$, so this property is preserved when we restrict to covers without short geodesics that don't lift to $T_\phi$. In other words, the probability that $Y^{(n)}_\phi$ has systole at least $\ell$, all the curves in the canonical pants decomposition have standard collars of width $w$ and $Y^{(n)}_\phi$ has a $2$-cover with spectral gap at least $\frac{1}{4}-\eta$, is asymptotically strictly positive.

So now suppose $\what{Y}_1$ and $\what{Y}_2$ are two covers of degree two of $Y_\phi^{(n)}$ that differ by a switch move. Moreover, let $f$ be a $\lambda_1(\what{Y}_1)$-eigenfunction on $\what{Y}_1$ and let $\alpha_i$ be the curve in which we need to switch to obtain $\what{Y}_2$. Using Proposition \ref{prp_costofswitching}, there exists a smooth function $f'$ that restricts to the $0$ function on both lifts of $\alpha_i$ and
\[
\frac{\langle \Delta f',f'\rangle}{\langle f',f'\rangle} \leq \lambda_1(\what{Y}_2) + \eta.
\]
Here we've used both the assumption on the collar widths and the systole of $Y_\phi^{(n)}$ and the relations between our constants (Section \ref{sec_constants}). Because $f'$ cancels on the lifts of $\alpha_i$, it defines a function on the cover $\what{Y}_2$ and hence
\[
\lambda_1(\what{Y}_2) \leq \lambda_1(\what{Y}_1) + \eta,
\]
by the variational characterization of $\lambda_1$, and we are done.
\end{proof}

%%%%%%%%%%%%%%%%%%%%%%%%%%%%%%%%%%%%%%%%%%%%%%%%%
%		B I B L I O G R A P H Y
%%%%%%%%%%%%%%%%%%%%%%%%%%%%%%%%%%%%%%%%%%%%%%%%%
\bibliography{bib}

\begin{thebibliography}{FBGMPP23}

\bibitem[AH22]{AranaHerrera}
Francisco Arana-Herrera.
\newblock Counting hyperbolic multigeodesics with respect to the lengths of
  individual components and asymptotics of {W}eil-{P}etersson volumes.
\newblock {\em Geom. Topol.}, 26(3):1291--1347, 2022.

\bibitem[AM99]{AdamsMorgan}
Colin Adams and Frank Morgan.
\newblock Isoperimetric curves on hyperbolic surfaces.
\newblock {\em Proc. Amer. Math. Soc.}, 127(5):1347--1356, 1999.

\bibitem[AW23]{AlonWei}
Noga Alon and Fan Wei.
\newblock The limit points of the top and bottom eigenvalues of regular graphs.
\newblock Preprint, \href{https://arxiv.org/abs/2304.01281}{arXiv:2304.01281},
  2023.

\bibitem[Bav96]{Bavard}
C.~Bavard.
\newblock Disques extr\'{e}maux et surfaces modulaires.
\newblock {\em Ann. Fac. Sci. Toulouse Math. (6)}, 5(2):191--202, 1996.

\bibitem[BC19]{BordenaveCollins}
Charles Bordenave and Beno\^it Collins.
\newblock Eigenvalues of random lifts and polynomials of random permutation
  matrices.
\newblock {\em Ann. of Math. (2)}, 190(3):811--875, 2019.

\bibitem[BCP21]{BCP_diameter}
Thomas Budzinski, Nicolas Curien, and Bram Petri.
\newblock On the minimal diameter of closed hyperbolic surfaces.
\newblock {\em Duke Math. J.}, 170(2):365--377, 2021.

\bibitem[Ber16]{Bergeron}
Nicolas Bergeron.
\newblock {\em The spectrum of hyperbolic surfaces}.
\newblock Universitext. Springer, Cham; EDP Sciences, Les Ulis, 2016.
\newblock Appendix C by Valentin Blomer and Farrell Brumley, Translated from
  the 2011 French original by Brumley [2857626].

\bibitem[Bon22]{Bonifacio}
J.~Bonifacio.
\newblock Bootstrapping closed hyperbolic surfaces.
\newblock {\em J. High Energy Phys.}, 93:Paper No. 093, 18, 2022.

\bibitem[BP23]{BakerPetri}
Elizabeth Baker and Bram Petri.
\newblock Statistics of finite degree covers of torus knot complements.
\newblock {\em Ann. H. Lebesgue}, 6:1213--1257, 2023.

\bibitem[Bus10]{Buser}
Peter Buser.
\newblock {\em Geometry and spectra of compact {R}iemann surfaces}.
\newblock Modern Birkh\"{a}user Classics. Birkh\"{a}user Boston, Ltd., Boston,
  MA, 2010.
\newblock Reprint of the 1992 edition.

\bibitem[CFKS87]{FKS}
H.~L. Cycon, R.~G. Froese, W.~Kirsch, and B.~Simon.
\newblock {\em Schr\"odinger operators with application to quantum mechanics
  and global geometry}.
\newblock Texts and Monographs in Physics. Springer-Verlag, Berlin, study
  edition, 1987.

\bibitem[Cha84]{Chavel}
Isaac Chavel.
\newblock {\em Eigenvalues in {R}iemannian geometry}, volume 115 of {\em Pure
  and Applied Mathematics}.
\newblock Academic Press, Inc., Orlando, FL, 1984.
\newblock Including a chapter by Burton Randol, With an appendix by Jozef
  Dodziuk.

\bibitem[Che70]{Cheeger}
Jeff Cheeger.
\newblock A lower bound for the smallest eigenvalue of the {L}aplacian.
\newblock {\em Problems in analysis ({P}apers dedicated to {S}alomon {B}ochner,
  1969)}, pages 195--199, 1970.

\bibitem[Che75]{Cheng}
S.~Y. Cheng.
\newblock Eigenvalue comparison theorems and its geometric applications.
\newblock {\em Math. Z.}, 143(3):289--297, 1975.

\bibitem[Con15]{Conder}
M.~D.~E. Conder.
\newblock Quotients of triangle groups acting on surfaces of genus 2 to 101.
\newblock
  \href{https://www.math.auckland.ac.nz/~conder/TriangleGroupQuotients101.txt}{math.auckland.ac.nz/$\sim$conder/TriangleGroupQuotients101.txt},
  2015.

\bibitem[CP23]{CornelissenPeyerimhoff}
G.~Cornelissen and N.~Peyerimhoff.
\newblock {\em Twisted isospectrality, homological wideness, and isometry---a
  sample of algebraic methods in isospectrality}.
\newblock SpringerBriefs in Mathematics. Springer, Cham, [2023] \copyright
  2023.

\bibitem[Dix69]{Dixon}
J.~D. Dixon.
\newblock The probability of generating the symmetric group.
\newblock {\em Math. Z.}, 110:199--205, 1969.

\bibitem[DM24]{DongMcKenzie}
Dingding Dong and Theo McKenzie.
\newblock Arbitrary spectral edge of regular graphs.
\newblock Preprint, \href{https://arxiv.org/abs/2412.09570}{arXiv:2412.09570},
  2024.

\bibitem[FBGMPP23]{FBGMPP}
Maxime Fortier~Bourque, \'{E}mile Gruda-Mediavilla, Bram Petri, and Mathieu
  Pineault.
\newblock Two counterexamples to a conjecture of {C}olin de {V}erdi\`{e}re on
  multiplicity.
\newblock Preprint, \href{https://arxiv.org/abs/2312.03504}{arXiv:2312.03504},
  2023.

\bibitem[FBP23]{FBP_LP}
Maxime Fortier~Bourque and Bram Petri.
\newblock Linear programming bounds for hyperbolic surfaces.
\newblock Preprint, \href{https://arxiv.org/abs/2302.02540}{arXiv:2302.02540},
  2023.

\bibitem[Gag80]{Gage}
M.~E. Gage.
\newblock Upper bounds for the first eigenvalue of the {L}aplace-{B}eltrami
  operator.
\newblock {\em Indiana Univ. Math. J.}, 29(6):897--912, 1980.

\bibitem[GAP22]{GAP}
The GAP~Group.
\newblock {\em {GAP -- Groups, Algorithms, and Programming, Version 4.12.2}},
  2022.

\bibitem[GLMST21]{GLMST}
Clifford Gilmore, Etienne Le~Masson, Tuomas Sahlsten, and Joe Thomas.
\newblock Short geodesic loops and {$L^p$} norms of eigenfunctions on large
  genus random surfaces.
\newblock {\em Geom. Funct. Anal.}, 31(1):62--110, 2021.

\bibitem[HM23]{HideMagee}
Will Hide and Michael Magee.
\newblock Near optimal spectral gaps for hyperbolic surfaces.
\newblock {\em Ann. of Math. (2)}, 198(2):791--824, 2023.

\bibitem[Jen81]{Jenni}
F.~Jenni.
\newblock Ueber das spektrum des {L}aplace-operators auf einer schar konmpakter
  {R}iemannscher fl{\"a}chen.
\newblock PhD thesis, University of Basel, 1981.

\bibitem[Joh17]{Arb}
F.~Johansson.
\newblock Arb: efficient arbitrary-precision midpoint-radius interval
  arithmetic.
\newblock {\em IEEE Transactions on Computers}, 66:1281--1292, 2017.

\bibitem[Kee74]{Keen}
Linda Keen.
\newblock Collars on {R}iemann surfaces.
\newblock In {\em Discontinuous groups and {R}iemann surfaces ({P}roc. {C}onf.,
  {U}niv. {M}aryland, {C}ollege {P}ark, {M}d., 1973)}, volume No. 79 of {\em
  Ann. of Math. Stud.}, pages 263--268. Princeton Univ. Press, Princeton, NJ,
  1974.

\bibitem[KMP24]{KravchukMazacPal}
Petr Kravchuk, Dalimil Maz\'a\v{c}, and Sridip Pal.
\newblock Automorphic spectra and the conformal bootstrap.
\newblock {\em Commun. Am. Math. Soc.}, 4:1--63, 2024.

\bibitem[Liu22]{Liu}
Mingkun Liu.
\newblock Length statistics of random multicurves on closed hyperbolic
  surfaces.
\newblock {\em Groups Geom. Dyn.}, 16(2):437--459, 2022.

\bibitem[LM23]{HideLouderMagee}
Larsen Louder and Michael Magee.
\newblock Strongly convergent unitary representations of limit groups.
\newblock Preprint, \href{https://arxiv.org/abs/2210.08953}{arXiv:2210.08953}.
  With an appendix by Will Hide and Michael Magee, 2023.

\bibitem[LP10]{LinialPuder}
Nati Linial and Doron Puder.
\newblock Word maps and spectra of random graph lifts.
\newblock {\em Random Structures Algorithms}, 37(1):100--135, 2010.

\bibitem[Mag24a]{Magee}
Michael Magee.
\newblock The limit points of the bass notes of arithmetic hyperbolic surfaces.
\newblock Preprint, \href{https://arxiv.org/abs/2403.00928}{arXiv:2403.00928},
  2024.

\bibitem[Mag24b]{Magee_survey}
Michael Magee.
\newblock Strong convergence of unitary and permutation representations of
  discrete groups.
\newblock to appear, 2024+.

\bibitem[Mat05]{Matsuzaki}
Katsuhiko Matsuzaki.
\newblock Isoperimetric constants for conservative {F}uchsian groups.
\newblock {\em Kodai Math. J.}, 28(2):292--300, 2005.

\bibitem[Mat24]{Mathien}
Joffrey Mathien.
\newblock Diameter of a new model of random hyperbolic surfaces.
\newblock Preprint, \href{https://arxiv.org/abs/2403.01925}{arXiv:2403.01925},
  2024.

\bibitem[Mir13]{Mirzakhani1}
Maryam Mirzakhani.
\newblock Growth of {W}eil-{P}etersson volumes and random hyperbolic surfaces
  of large genus.
\newblock {\em J. Differential Geom.}, 94(2):267--300, 2013.

\bibitem[Mir16]{Mirzakhani2}
Maryam Mirzakhani.
\newblock Counting mapping class group orbits on hyperbolic surfaces.
\newblock Preprint, \href{https://arxiv.org/abs/1601.03342}{arXiv:1601.03342},
  2016.

\bibitem[MNP22]{MageeNaudPuder}
Michael Magee, Fr\'ed\'eric Naud, and Doron Puder.
\newblock A random cover of a compact hyperbolic surface has relative spectral
  gap {$\frac{3}{16}-\varepsilon$}.
\newblock {\em Geom. Funct. Anal.}, 32(3):595--661, 2022.

\bibitem[MP23]{MageePuder}
Michael Magee and Doron Puder.
\newblock The asymptotic statistics of random covering surfaces.
\newblock {\em Forum Math. Pi}, 11:Paper No. e15, 51, 2023.

\bibitem[MR03]{MaclachlanReid}
Colin Maclachlan and Alan~W. Reid.
\newblock {\em The arithmetic of hyperbolic 3-manifolds}, volume 219 of {\em
  Graduate Texts in Mathematics}.
\newblock Springer-Verlag, New York, 2003.

\bibitem[MT23]{MageeThomas}
Michael Magee and Joe Thomas.
\newblock Strongly convergent unitary representations of right-angled artin
  groups.
\newblock Preprint, \href{https://arxiv.org/abs/2308.00863}{arXiv:2308.00863},
  2023.

\bibitem[Nic94]{Nica}
Alexandru Nica.
\newblock On the number of cycles of given length of a free word in several
  random permutations.
\newblock {\em Random Structures Algorithms}, 5(5):703--730, 1994.

\bibitem[PZ24]{PuderZimhoni}
Doron Puder and Tomer Zimhoni.
\newblock Local statistics of random permutations from free products.
\newblock {\em Int. Math. Res. Not. IMRN}, (5):4242--4300, 2024.

\bibitem[Sar23]{Sarnak_Chern_lectures}
Peter Sarnak.
\newblock Spectra of locally uniform geometries.
\newblock Chern Lectures, U.C. Berkeley. Available at:
  \url{http://publications.ias.edu/sarnak/paper/2728}, 2023.

\bibitem[SU13]{StrohmaierUski}
A.~Strohmaier and V.~Uski.
\newblock An algorithm for the computation of eigenvalues, spectral zeta
  functions and zeta-determinants on hyperbolic surfaces.
\newblock {\em Comm. Math. Phys.}, 317:827--869, 2013.

\bibitem[Tak77]{Takeuchi}
Kisao Takeuchi.
\newblock Arithmetic triangle groups.
\newblock {\em J. Math. Soc. Japan}, 29(1):91--106, 1977.

\bibitem[{The}21]{sagemath}
{The Sage Developers}.
\newblock {\em {S}ageMath, the {S}age {M}athematics {S}oftware {S}ystem
  ({V}ersion 9.3)}, 2021.
\newblock \url{https://www.sagemath.org}.

\end{thebibliography}
\bibliographystyle{alpha}

\end{document}